\documentclass[12pt]{amsart}
\usepackage{a4}
\usepackage{graphicx}
\usepackage{amsmath,amssymb,amsthm}
\usepackage{geometry}                
\usepackage{epstopdf}
\title{A L'hospital's rule for multivariable functions}
\author{Gary R. Lawlor}

\theoremstyle{plain}
\newtheorem{thm}{Theorem}[section]

\theoremstyle{definition}

\newtheorem{exm}[thm]{Example}

\def \RR {\mathbb{R}}     \def \RRR {\mathbb{R}^}         \def \xxx {{\bf x}} \def \vvv {{\bf v}}    \def \ppp {{\bf p}}            \def \mcN {{\mathcal{N}}}

\def \Lhos   {L'h\^{o}pital's }
\def \Lho {L'h\^{o}pital }

\begin{document}

\maketitle
\section{Introduction}

Zero divided by zero is arguably the most important concept in calculus, as it is the gateway to the world of differentiation, as well as (via the fundamental theorem of calculus) the calculation of integrals.  Organized methods of finding the right answers to zero-over-zero problems were developed by Newton and Leibniz and those who followed.  The concept of limit was finally made rigorous long afterward by Cauchy and others.  

We are all sternly informed in grade school that zero over zero is impossible and must be shunned.  A more accurate statement, however, (once we are ready to understand it) is that it is impossible to divide zero by zero \emph{out of context}. 

All other arithmetic operations are independent of context, a rather remarkable fact.  So six divided by three equals two, no matter whether you are sharing cookies with a couple of friends or making a fraction of a recipe of bread.  Not so with zero divided by zero: if you take a photograph of two cars and observe that in the picture both cars travel zero feet in zero seconds, that is no help in comparing the cars' velocities.

But if you can find some contextual information to a zero-over-zero problem and take a limit, you are back in business.  
As an example, if a problem traces back to the function 
$$f(x)=\frac{(3+x)^2-9}{x}$$
for $x$ near 0,
then we have at least four options for deciding what $f(x)$ ``ought to equal'' when $x=0$, namely plugging in nearby values and guessing, using algebra and cancellation, applying L'hospital's rule, or recognizing that the limit of $f$ as $x$ approaches zero represents the derivative of $x^2$ at $3$.

But what if the context of a zero-over-zero problem comes from a  function of several variables?  Is there a multivariable version of L'hospital's rule?  As simple a limit as
$$\lim_{{\scriptscriptstyle\begin{array}{cc}(x,y)\to(0,0) \\ x \ne y\end{array}}}\frac{x-y}{\sin x-\sin y}$$ can be rather puzzling to consider, while something like
$$\lim_{(x,y,z)\to(0,0,0)}\frac{7x^2yz^5+xy^3-3x^4yz}{x^8+x^2y^2z^4+(y-x^3+z^2)^2+z^6-xy^3z^5}$$
is daunting indeed!

A perusal of your multivariable calculus book will not likely help very much.  At best you'll find some interesting individual examples worked out, with a caution not to try such things at home --- at least not until you are properly warned against relying only upon individual lines through the origin!  You can verify, for example, that if we set $x=at$, $y=bt$ and $z=ct$, the above limit as $t\to 0$ always works out to be zero, regardless of the values of $a$, $b$ and $c$.  But that is certainly not the whole story, as we shall demonstrate at the end of the paper.

We seem largely to have overlooked the topic of finding an organized approach for resolving zero-over-zero limits of multivariable functions.

The two papers \cite{dobsic} and \cite{young}, p.~71, both handle the specific situation of a two-variable indeterminate limit resolvable by taking the mixed second derivative $\partial^2/\partial x\partial y$ of the numerator and denominator functions.  

The paper \cite{finekass} has a version using first-order derivatives, but the theorem's usefulness turns out to be limited, as we discuss after the proof of Theorem \ref{lhospuni} below.

There are some papers with a good treatment of the indeterminate limit of a quotient of a vector-valued function over a real-valued function, but these papers concern functions of a single variable.  See \cite{rosenholtz}, \cite{alb}, \cite{popa} and \cite{wazewski} for \Lho-style theorems of this type.
 
\vskip 0.1in
We now present a method for resolving many multivariable indeterminate limits.

\section{Nonisolated singular points}

A key aspect of a zero-over-zero limit of a multivariable function turns out to be the question of whether the singular point is isolated or not.  Perhaps surprisingly, the nonisolated case turns out to yield the nicer theorem.  

If the singular point $\ppp$ is not an isolated zero of the denominator function $g$, then we understand the limit to be taken over points where $g\ne 0$.  In order to hope for a limit to exist, the zero set of $g$ must be contained in the zero set of $f$ near $\ppp$; otherwise the function will blow up near zeroes of $g$ that are not zeroes of $f$.  This hypothesis about the zero sets does not appear explicitly in Theorem \ref{lhospuni}, but it is implied by the hypotheses that are there.

\begin{thm}[Multivariable \Lhos rule]  \label{lhospuni}
Let ${\mathcal{N}}$ be a neighborhood in $\RRR 2$ containing a point $\ppp$ at which two differentiable functions $f:\mcN\to\RR$ and $g:\mcN\to\RR$ are zero.

Set 
$$C=\{\xxx\in \mcN: f(\xxx)=g(\xxx)=0\},$$
and suppose that $C$ is a smooth curve through $\ppp$.

Suppose there exists a vector $\vvv$ not tangent to $C$ at $\ppp$ such that the directional derivative $D_\vvv g$ of $g$ in the direction of $\vvv$ is never zero within $\mcN$.

More generally, if $C$ consists of a union of two or more smooth curves through $\ppp$, suppose that for each component $E_i$ of $\mcN\setminus C$ we can find a vector $\vvv_i$, not tangent at $\ppp$ to any of the curves comprising $C$, such that $D_{\vvv_i} g \ne 0$ on $E_i$.

Then
$$\lim_{(x,y)\to\ppp}\frac{f(x,y)}{g(x,y)}=\lim_{\scriptscriptstyle{\begin{array}{c}(x,y)\to\ppp \\ (x,y)\in E_i \end{array}}}\frac{D_{\vvv_i} f(x,y)}{D_{\vvv_i} g(x,y)}$$
if the latter limits exist and are equal for all $i$.  Each limit is assumed to be taken over the domain of points where the denominator is nonzero, and we assume in each case that $\ppp$ is a limit point of that domain.

The same theorem holds for real-valued functions of $n$ variables, with $C$ a union of hypersurfaces.
\end{thm}

\begin{proof}  Suppose first that $C$ is just a single smooth curve through $\ppp$.  Let $L$ be the line through $\ppp$ in the direction of $\vvv$.  
Since $\vvv$ is not tangent to $C$ at $\ppp$, parallel lines near $L$ also cross through $C$ near $\ppp$; see Figure 1.

\begin{figure}[h]
\label{transverse}
\includegraphics[scale=0.7]{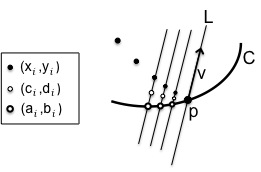}
\vskip -0.2in
\caption{Sequences of points approaching a singularity}
\end{figure}

Take a sequence $(x_i,y_i)$ of points converging to $\ppp$ with $g(x_i,y_i)\ne 0$ for all $i$.  Let $(a_i,b_i)$ be a corresponding sequence of points on $C$, also converging to $\ppp$, such that $(x_i,y_i)-(a_i,b_i)$ is a scalar multiple of $\vvv$ for all $i$.  Now for each $i$, as in the proof of the single-variable \Lhos rule, apply the Cauchy mean value theorem to the function $f/g$ restricted to the line in the direction of $\vvv$.  Recalling that $D_\vvv g\ne 0$ (it would be ironic to make a division-by-zero error in this proof!) we find that there is a point $(c_i,d_i)$ in between $(x_i,y_i)$ and $(a_i,b_i)$ such that
$$\frac{D_\vvv f(c_i,d_i)}{D_\vvv g(c_i,d_i)}=\frac{f(x_i,y_i)-f(a_i,b_i)}{g(x_i,y_i)-g(a_i,b_i)}=\frac{f(x_i,y_i)}{g(x_i,y_i)}.$$
But now since the points $(c_i,d_i)$ also converge to $\ppp$, the desired result follows.

Now if $C$ consists of more than one curve through $\ppp$ then the same proof can be applied, with the domain restricted to any component $E_i$.  The result is that the limit of $f/g$ with domain $E_i$ obeys the desired formula; since this is true (and gives the same value) for all of the finitely many components $E_i$, then the overall limit obeys the same formula.
\vskip 0.05in

The proof for $\RRR n$ goes through unchanged.
\end{proof}
 
\vskip 0.01in
We note here the importance of taking the directional derivative with respect to a fixed direction $\vvv$.  This works in the proof above because all lines passing near $\ppp$ in the direction of $\vvv$ do pass through the common zero set of $f$ and $g$.  A different strategy might be to take the directional derivative always in the direction straight out from a singular point.  Indeed, the paper \cite{finekass}, which handles only isolated singular points, does precisely this.  The result is a true theorem, with proof similar to that above, but rarely is the expression obtained after differentiating any simpler than the original; the degree of a rational function, for example, does not decrease.

\vskip 0.1in

\begin{exm} \label{sinxsiny}
Let
$$h(x,y)=\frac{f(x,y)}{g(x,y)}=\frac {x-y}{\sin x-\sin y}.$$
Then near the origin the zero sets of both numerator and denominator consist of the line $y=x$.  Choose, say, $\vvv=(1,0)$, so that the directional derivative to be taken is just the partial derivative by $x$.  Then
$$\lim_{{\scriptscriptstyle\begin{array}{cc}(x,y)\to(0,0) \\ x \ne y\end{array}}}\frac{x-y}{\sin x-\sin y}=\lim\frac{f_x}{g_x}=\lim\frac{1}{\cos x}=1.$$
\end{exm}

\begin{exm} \label{threevars}
Let
$$h(x,y,z)=\frac{f(x,y,z)}{g(x,y,z)}=\frac {\sin z-\sin(x^2+y^2)}{\tan(z-x^2)-\tan(y^2)}.$$
Then near the origin the zero sets of both numerator and denominator consist of the paraboloid $z=x^2+y^2$.  Choose, $\vvv=(0,0,1)$, so that the directional derivative to be taken is just the partial derivative by $z$.  The hypotheses of Theorem \ref{lhospuni} are satisfied, and
$$\lim_{\scriptstyle(x,y,z)\to(0,0,0)}\frac {\sin z-\sin(x^2+y^2)}{\tan(z-x^2)-\tan(y^2)}=\lim_{\scriptstyle(x,y,z)\to(0,0,0)}\frac{\cos z}{\sec^2(z-x^2)}=1.$$
\end{exm}

\begin{exm} Let
$$h(x,y)=\frac{f(x,y)}{g(x,y)}=\frac {x^2-y^2}{\cos x-\cos y}.$$
Near the origin, numerator and denominator are zero when $y=\pm x$.  Choosing $\vvv_1=(1,0)$, we note that $D_{\vvv_1} g=-\sin x$ is zero on the $y$ axis, which intersects neither the open ``east component'' $E_1$ of the plane bounded by the northeast and southeast arms of the zero set of $g$, nor the opposite ``west component.''  We calculate
$$\lim_{(x,y)\to(0,0)}h(x,y)=\lim _{(x,y)\to(0,0)}-\frac{2x}{\sin x}=-2$$
for $(x,y)$ restricted to $E_1$.  

For the north and the south components we use $\vvv_2=(0,1)$ and obtain the same limit, so the overall limit is $-2$.
\end{exm}


\section{Isolated singular points} \label{isolated}

Now we move to the case where the denominator function $g$ is (or approaches) zero at an isolated point $\ppp$.  We prescribe the following algorithm, not always conclusive, for deciding the value of such a limit.  The basic idea is to use a change of variables that makes the denominator into a sum of squares.

\vskip 0.1in \noindent
{\bf Algorithm:} To attempt to resolve a limit in $\RRR n$ of the form
$$\lim_{\xxx\to\ppp}\frac{f(\xxx)}{g(\xxx)}$$
where $g(\xxx)\ne 0$ on a deleted neighborhood of $\ppp$ but $g(\xxx)$ equals (or approaches) 0 at the point $\ppp$:
\begin{enumerate}
\item[Step 1:] Make a simple, preliminary substitution, such as holding all but one variable constant, and find the resulting restricted limit at $\ppp$.  

\begin{itemize}
\item If the restricted limit equals 0, proceed to Step 2.

\item If the restricted limit does not exist, then of course the unrestricted limit does not exist, either.  

\item If the restricted limit exists and equals $\ell\ne 0$, replace the numerator $f(\xxx)$ by $f(\xxx)-\ell g(\xxx)$, so that the restricted limit will now equal zero.  

\vskip 0.1in Thus, the two remaining possibilities are that the (unrestricted) limit is zero or undefined.  
\end{itemize}

\item[Step 2:] See whether a separation technique such as in Example \ref{exone} below resolves the limit.  If not, proceed to Step 3.

\item[Step 3:] Write the denominator $g(\xxx)$ as a sum of squares $u_1(\xxx)^2+\cdots+u_n(\xxx)^2$ plus (optionally) another function $v(\xxx)$.  Each $u_i(\xxx)$ must approach zero as $\xxx\to\ppp$.  

The function $v$, if present, will need to be less significant near $\ppp$ than the sum of $u_i(\xxx)^2$, as described further below.

\item[Step 4:] To prove the limit does not exist, try the substitution $u_i(\xxx)=m_i t$ for all $i$.  If this defines, for each appropriate choice of constants $\{m_i\}$, a curve through $\ppp$, and if the resulting restricted limit does not exist or depends on the values of $\{m_i\}$, then the original limit does not exist.  

\item[Step 5:] To prove the limit is zero, make a polar or spherical coordinate substitution $u_i(\xxx)=h_i(\theta_1,\ldots,\theta_{n-1})\rho$, where the sum of squares of the functions $h_i$ is identically 1.  If the numerator $f(\xxx)$ and the residual denominator term $v(\xxx)$ can both be written as linear combinations of $\rho^{\alpha_j}$ times bounded functions, with the powers $\alpha_j$ all strictly greater than 2, then we can divide top and bottom by $\rho^2$, leaving a positive power of $\rho$ times a bounded function.  Since $\rho\to 0$ as $\xxx\to {\bf 0}$, the overall limit is zero.
\end{enumerate}

\vskip 0.2in
\begin{exm} \label{exone}
We begin with an example that can be resolved by separating variables:
$$\lim_{(x,y)\to (0,0)} \,\,\frac{x^2+\sin^4 y}{\sin^2 x+y^4}.$$
Setting, for example, $y=0$ and applying the single-variable L'hospital's rule we get a restricted limit of 1.  Thus, according to Step 1 we should try to prove that the following new limit equals zero:
$$\lim_{(x,y)\to (0,0)} \,\,\frac{x^2+\sin^4 y-\sin^2 x-y^4}{\sin^2 x+y^4}.$$
Now this equals
$$\lim_{(x,y)\to (0,0)} \,\,\frac{x^2 - \sin^2 x }{\sin^2 x+y^4} + \lim_{(x,y)\to (0,0)} \,\,\frac{ \sin^4 y -y^4}{\sin^2 x+y^4},$$
which is bounded in absolute value by
$$\lim_{(x,y)\to (0,0)} \,\,\frac{|x^2 - \sin^2 x| }{\sin^2 x} + \lim_{(x,y)\to (0,0)} \,\,\frac{ |-\sin^4 y +y^4|}{y^4}.$$
L'hospital's rule shows each of these limits to be zero, as desired.
\end{exm}
\begin{exm}
We continue with a simple example, standard in textbooks:
$$\frac{f(x,y)}{g(x,y)}=\frac{xy}{x^2+y^2}.$$
The singularity at the origin is isolated, and the preliminary substitution $y=0$ gives a restricted limit of 0.

The denominator is already a sum of squares of polynomials, with $u_1(x,y)=x$ and $u_2(x,y)=y$.  As prescribed in step 3 of the algorithm, we set
$$x=m_1 t \mbox{ and }y=m_2 t,$$
resulting in the limit
$$\lim_{t\to 0} \,\,\frac{m_1m_2 t^2}{m_1^2t^2+m_2^2t^2}=\frac{m_1m_2}{m_1^2+m_2^2};$$
since this limit depends on $m_1$ and $m_2$, the original limit does not exist.
\end{exm}

The point commonly made in connection with the above function is that although you can prove
that a limit \emph{does not exist} by restricting to \emph{straight lines} through the origin in the domain, 
this method does not suffice to prove that a limit \emph{does exist}. 
A typical example cited is the following, which we will analyze first by lines through the origin and then through the algorithm above.

\begin{exm} Consider
$$\lim_{(x,y)\to(0,0)}\frac{x^2 y}{x^4+y^2}.$$

\begin{figure}[h]
\label{hithere}
\includegraphics[scale=0.82]{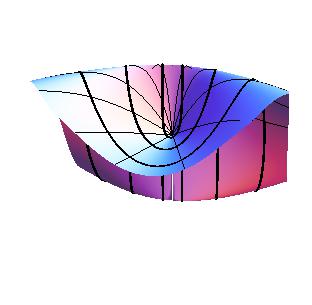}
\vskip -0.9in 
\caption{$z=x^2y/(x^4+y^2)$}
\end{figure}

If we were to set $x=r\cos\theta$ and $y=r\sin\theta$, we would get
$$\frac{r^3\cos^2\theta\sin\theta}{r^4\cos^4\theta+r^2\sin^2\theta}=\frac{r\cos^2\theta\sin\theta}{r^2\cos^4\theta+\sin^2\theta}$$
when $r\ne 0$.  Holding $\theta$ constant and taking the limit as $r\to 0$ we always get a (restricted) limit of 0.  This is reflected by the lighter curves in the figure (on which $\theta$ is constant) all converging to one center point. 

We might think, then, that the overall limit of $g$ is zero.  As the textbooks point out, this is not the case, as shown by the substitution $y=x^2$.  Our algorithm prescribes essentially the same substitution; with
$$u_1(x,y)=x^2 \quad \mbox{ and } \quad u_2(x,y)=y,$$
we set
$$x^2=m_1 t \quad \mbox{ and }\quad y= m_2 t.$$
The limit as $t\to 0$ of the restricted function is now
$$\frac{m_1 m_2}{m_1^2+m_2^2}.$$
This is different for different values of $m$, so the original limit does not exist.  This is reflected by the heavier curves in the figure (on which $r$ is constant) having to drape up over the hump.
\end{exm}
\vskip 0.1in

\begin{exm}  \label{rational} Let
$$h(x,y)=\frac{x^3y^3}{x^6+y^4}.$$

We set
$$x^3=m_1 t \quad\mbox{ and }\quad y^2=m_2 t.$$
Then
$$h=\frac{m_1m_2^{3/2}t^{5/2}}{(m_1^2+m_2^2)t^2}=\frac{m_1m_2^{3/2}}{(m_1^2+m_2^2)}t^{1/2},$$
which goes to zero at $t\to 0$.  Suspecting that the limit is zero, we move to step 4.

Set $x^3=r\cos\theta$ and $y^2=r\sin\theta$.  Then
$$|h|=\frac{r^{5/2}|\cos\theta||\sin\theta|^{3/2}}{r^2}\le r^{1/2},$$
so that the limit is zero.

\end{exm}
{\bf Note:} The substitution above does cover all values of $x$ and $y$, whether positive or negative.  The substitution $y^2=r\sin\theta$ is not one-to-one; it has to use the same values of $\theta \in [0,\pi]$ for positive and negative values of $y$.  But since the polar limit is zero and since $\|(x,y)\|$ small implies $r$ is small, the original limit is also zero.

\vskip 0.1in
A fancier example features more than two square terms in the denominator.
\begin{exm}
Let
$$h(x,y)=\frac{f(x,y)}{g(x,y)}=\frac{x^3 y^2}{x^6+x^2y^2+y^6}.$$

\begin{figure}[h]
\label{graph2}
\includegraphics[scale=0.62]{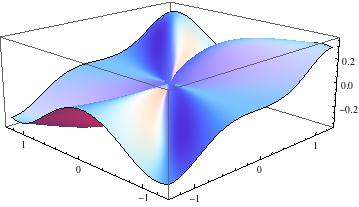}
\caption{$z=x^3y^2/(x^6+x^2y^2+y^6)$}
\end{figure}

The algorithm asks for just two square terms from the denominator.  We try different pairs.

First let us set
$$u_1(x,y)=x^3, \quad u_2(x,y)=y^3, \quad\mbox{ and }\quad v(x,y)=x^2y^2.$$
Then (setting the $m_i$'s equal to 1 for quicker analysis) we make the substitution
$$x^3= t= y^3$$
and obtain
$$h=\frac{t^{5/3}}{ t^2+t^{4/3}+t^2}=\frac{t^{1/3}}{1+2t^{2/3}},$$
which goes to zero.  But if we try to verify the zero limit by setting $x^3=r\cos\theta$ and $y^3=r\sin\theta$ we get
$$\frac{r^5\cos^3\theta\sin^2\theta}{r^2+r^{4/3}(\cos\theta\sin\theta)^{2/3}}    ,$$
which is inconclusive.  Note that with this substitution, the power of $t$ coming from the side term $v(x,y)$ was smaller than 2.

Next we try 
$$u_1(x,y)=x^3,\quad u_2(x,y)=xy,\quad \mbox{ and }\quad v(x,y)=y^6.$$
The substitution
$$x^3=m_1 t \quad\mbox{ and }\quad xy = m_2 t$$
gives
$$h=\frac{(x)(xy)^2}{x^6+x^2y^2+y^6}\le \frac{(x)(xy)^2}{x^6+x^2y^2}$$
$$= \frac{m_1^{1/3}t^{1/3}m_2^2t^2}{m_1^2t^2+m_2^2t^2},$$ which is bounded by a constant times a positive power of $t$, and so goes to zero.  We verify this zero limit by setting
$$x^3=r\cos\theta,\quad xy=r\sin\theta$$
$$|h|=\frac{r^{7/3}|\cos^{1/3}\theta\sin^2\theta|}{r^2+y^6}\le\frac{r^{7/3}}{r^2}=r^{1/3}.$$
The limit of $r^{1/3}$ is zero, so that the original function has a zero limit at the origin.

\end{exm}

One final example illustrates how we may replace functions by Taylor polynomial approximations, and also use dominant square terms to eliminate non-dominant terms.
\begin{exm} \label{great}
Let
$$h(x,y)=\frac{2-2\cos(x^2y^2)}{x^{10}+x^6y^2+y^6-x^9\sin  y}.$$

\begin{figure}[h]
\label{graph3}
\includegraphics[scale=0.72]{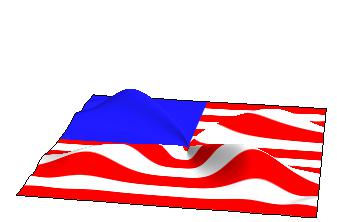}
\caption{The function of Example \ref{great}}
\end{figure}

Define
\begin{align*}g_1(x,y) &= x^{10}+x^6y^2+y^6 \\ g_2(x,y) &=-x^9\sin  y \\
\tilde{g_2}(x,y) &= -x^9y \\
 f(x,y) &=2-2\cos(x^2y^2)\\ \tilde f(x,y) &=x^4y^4. \end{align*}
Now by setting $u=x^2y^2$ and using the single-variable \Lhos rule we see that
$$\lim_{(x,y)\to (0,0)}\frac{\tilde f(x,y)}{f(x,y)}=1;$$
similarly we verify
$$\lim_{(x,y)\to (0,0)}\frac{\tilde {g_2}(x,y)}{g_2(x,y)}=1.$$
Then if we show that the limits of $\tilde{g_2}/g_1$ and $\tilde{f}/g_1$ are zero we will also know that the limit of $g_2/g_1$ is zero, and the limit of $(g_1+g_2)/g_1$ is 1; thus
$$\lim h=\lim \frac{f}{g_1+g_2}=\lim\frac{f}{g_1+g_2}\,\frac{g_1+g_2}{g_1}\,\frac{\tilde{f}}{f}=\lim\frac{\tilde{f}}{g_1}=0.$$
\end{exm}
To analyze $\tilde{g_2}/g_1$, set 
$$x^3 y=r\cos\theta \mbox{ and }x^5=r\sin\theta.$$
Then
$$|\frac{\tilde{g_2}(x,y)}{g_1(x,y)} | = \frac{r |\cos\theta| \,r^{6/5}|\sin\theta|^{6/5}}{r^2+y^6}$$
$$\le \frac{r^{11/5}}{r^2}=r^{1/5},$$
which goes to zero as $r\to 0$.  Thus, 
$$\lim_{(x,y)\to(0,0)}\frac{\tilde{g_2}(x,y)}{g_1(x,y)}=0.$$

To analyze $\tilde{f}/g_1$, 
this time we need to highlight a different pair of square terms from the denominator; set
$$x^3 y=r\cos\theta \mbox{ and }y^3=r\sin\theta.$$
Then
$$|\frac{\tilde{f}(x,y)}{g_1(x,y)}|=\frac{r^{4/3}|\cos\theta|^{4/3}\,r^{8/9}|\sin\theta|^{8/9}}{r^2+x^{10}}$$
$$\le \frac{r^{20/9}}{r^2}=r^{2/9}\to 0.$$

One technicality is that some of the limits being combined here, such as the limit of $\tilde{f}/f$, are undefined when $x=0$ or $y=0$.  For this, note that if the overall limit of $h$ is zero when restricted to each of \emph{finitely many} domains then the limit over the union of those domains is also zero.  Clearly the limit of $h$ is zero when restricted to $x=0$ or to $y=0$, so the analysis is complete.

\section{Unfinished business}
We promised in the introduction to complete the analysis of the limit
$$\lim_{(x,y,z)\to(0,0,0)}\frac{7x^2yz^5+xy^3-3x^4yz}{x^8+x^2y^2z^4+(y-x^3+z^2)^2+z^6-xy^3z^5}.$$
We noted there that all restricted limits along lines through the origin are zero.  To analyze further, set 
\begin{align*} u_1 &= x^4\\ u_2 &= xyz^2\\ u_3&= y-x^3+z^2 \\ u_4 &=z^3\end{align*}
we obtain
$$\lim\frac{7u_1^{1/4}u_2u_4+u_1^{1/4}(u_3+u_1^{3/4}-u_4^{2/3})^3-3u_1(u_3+u_1^{3/4}-u_4^{2/3})u_4^{1/3}}{u_1^2+u_2^2+u_3^2+u_4^2-u_2u_4(u_3+u_1^{3/4}-u_4^{2/3})^2}.$$
Now the non-square denominator term and all numerator terms but one are of degree greater than 2.  But the last numerator term is of degree 2.  

The substitution in Step 4 of the algorithm of section \ref{isolated} would result in an overdetermined system.  Since the term of degree 2 in the numerator only involves $u_1$, $u_3$ and $u_4$, we leave out $u_2$ in determining $t$.  Also, to eliminate fractional exponents we will use $t^{12}$ in place of $t$.  Choosing $m_1=m_3=m_4=1$ we set
$t^{12}=u_1=u_3=u_4$ and solve for $x$, $y$, and $z$ to obtain
\begin{align*}x&=t^3\\ z&=t^4 \\ y &= t^{12}+ t^9-t^8.   \end{align*}
Then the original limit translates to
$$\frac{3t^{24}+O(t^{25})}{3t^{24}+O(t^{25})},$$
whose limit is $1$.  So the original function is bounded in a neighborhood of $(0,0,0)$ (since no numerator term has degree less than 2), and it turns out that the function goes to zero except on a tiny thread through the origin.

\end{document}